\documentclass[a4paper,12pt]{amsart}
\usepackage{color}
\usepackage{amsmath}
\usepackage{amssymb}
\usepackage{amsthm}
\usepackage[mathscr]{eucal}
\usepackage{xcolor}

\numberwithin{equation}{section}

\newcommand{\id}{\operatorname{id}}
\newcommand{\R}{\operatorname{\mathbb{R}}}

\newcommand{\Laplace}{\Delta}

\newcommand{\Ga}{\operatorname{\Gamma}}
\newcommand{\la}{\operatorname{\lambda}}
\newcommand{\La}{\operatorname{\Lambda}}
\newcommand{\ka}{\operatorname{\kappa}}

\def\om{\omega}
\newcommand{\sd}{\kern-.1em\cdot\kern-.1em}

\def\3{\ss}

\def\tr{\operatorname{tr}}

\def\D{\partial}

\theoremstyle{plain}
\newtheorem{thm}{Theorem}[section]
\newtheorem{lem}[thm]{Lemma}
\newtheorem{prop}[thm]{Proposition}
\newtheorem{cor}[thm]{Corollary}

\newtheorem{question}[thm]{Question}

\theoremstyle{definition}
\newtheorem{defn}[thm]{Definition}

\theoremstyle{remark}
\newtheorem{rem}[thm]{Remark}

\begin{document}

\title{Improvements of upper curvature bounds}

\author{Alexander Lytchak}
%\address{Alexander Lytchak}
%	Mathematisches Institut,
%	Universit\"at K\"oln,
%	Weyertal 86--90, 50931, K\"oln, Germany}
%\email{alytchak@math.uni-koeln.de}

\author{Stephan Stadler}
%\address{Stephan Stadler}
%	Mathematisches Institut der Universit\"at M\"unchen, Theresienstr. 39, D-80333 M\"unchen, Germany}
%\email{stadler@math.lmu.de}



\begin{abstract}
	We show that any space with a positive upper curvature bound has in a small neighborhood of any point a  closely related metric
	with a negative upper curvature bound.
\end{abstract}

\maketitle

\section{Introduction}

\subsection{Main result}
There is a  significant difference between the existence of global \emph{non-positive}  upper curvature bounds and the existence of \emph{some} upper curvature bound $\kappa \in \R$.
For  instance, any complete non-positively curved space is
aspherical. On the other hand, \emph{any} simplicial complex carries a metric
of curvature bounded from above by $1$ in the sense of Alexandrov, \cite{Berest}.

The main result  of this note confirms the expectation that in \emph{local} considerations the value of the upper curvature bound does not matter:

\begin{thm} \label{thm: main}
	For $\kappa \in \R$ let the metric space $(X,d)$ be CAT($\kappa$).
	Let $O=B_r(x)$ be an open ball of radius $r$ around $x$ in $X$.
	If $\kappa >0$ assume  $r<\frac {\pi} {2\sqrt \kappa}$.
Then there exists a complete CAT(-1) metric $d'$ on $O$, such that the identity map
	 $(O,d)\to (O,d')$ is locally bilipschitz.

	 \end{thm}

In particular, in many  questions concerning only local topological properties of CAT($\kappa $) spaces,
like most of \cite{Kleiner}, \cite{LN1}, \cite{LN2}, one may always assume $\kappa$ to be  $-1$.
\begin{rem}
 The metric $d'$ provided by the proof
of Theorem \ref{thm: main} has  the following additional properties. The distance to the central point $x$ in $(O,d')$ depends only on the distance to  $x$ in $(O,d)$.
Moreover, the tangent spaces at any point $y\in O$ with respect to both metrics $d$ and $d'$ are isometric. This  condition implies that  $(O,d')$ is geodesically complete if $(O,d)$ is locally geodesically complete, see \cite{LN1}.
\end{rem}

Our construction follows \cite{LS} and defines the metric $d'$ via a conformal change of the original metric $d$ with a sufficiently convex function. The same approach shows that any
CAT(0) space $(X,d)$ admits another CAT(0) metric which is locally \emph{negatively} curved, see Theorem \ref{thm: nonpos} below.  This result also applies  to everywhere branching   Euclidean buildings,
where some rigidity might have been expected, see \cite{KleinerLeeb}, \cite{Linus}.   However, in   Theorem \ref{thm: nonpos} the local negative curvature bound  needs  to tend to zero at infinity. Indeed,
 no conformal change of the flat Euclidean plane results in a complete Riemannian manifold of curvature $\leq -1$, \cite[Corollary 7.3]{uniform}.   Thus, the answer to the following  question,
 very natural in view of \cite{AB} and our Theorems \ref{thm: main}, \ref{thm: nonpos}, cannot be obtained by means of this paper:

\begin{question}
	Given a CAT(0) space $X$, does there exists a CAT(-1) metric on $X$ defining the same topology?
\end{question}

\subsection{Auxiliary results of  independent interest}
It has been shown in \cite{LS} that conformal changes with sufficiently convex functions preserve non-positive  curvature.
The  improvement of  the curvature bound is derived from the following
 analog of the formula expressing the curvature of a Riemannian manifold after a conformal change.

\begin{thm} \label{thm: confchange}
	For $c,C, \kappa, \lambda \in \R$, let $X$ be a CAT($\kappa$) space and let $f:X\to [c,C ]$ be a  Lipschitz continuous $\lambda$-convex function.   Further, let  $Y= e^{f} \cdot X$ denote the conformally equivalent space.
	\begin{itemize}
	\item   If $\ka  -4\lambda  \leq 0$ then $Y$ is CAT($\bar\kappa$) with
	$\bar \kappa 	= e^{-2C} \cdot  (\ka -4\lambda )$.
	
	\item   If $\kappa -4\lambda >0$ and $\lambda >0$ then $Y$ is
	CAT($\bar\kappa$) with
	 $\bar \kappa 	= e^{-2c} \cdot  (\ka -4\lambda )$.
	
	\end{itemize}

\end{thm}

 We refer to \cite{LS} and Subsection \ref{subsec: confchnage} below
for the definition and basic properties of conformally changed spaces and
recall that a  function $f:X\to\R$ is called $\lambda$-{\em convex
}
if  the function
$$ t\to f\circ\gamma (t)-\frac \lambda 2  \cdot  t^2$$
is convex, for any unit speed geodesic $\gamma$ in $X$.

Besides the theory of minimal discs, Theorem \ref{thm: confchange} relies  on a generalization of the classical observation that the restriction of a convex function to a harmonic map is subharmonic,
\cite{Ishihara}, \cite{KS},  \cite{Chen}, \cite{Fuglede}.  Here we derive the following natural  extension to semi-convex functions
 as a direct consequence of the properties of gradient flows of such functions.

\begin{thm} \label{thm: fuglede}
	Let $\Omega$ be a domain in  a Euclidean space $\R^n$ and let $u:\Omega \to X$ be a   harmonic map into a CAT($\kappa$) space $X$. Let $f:X\to \R$ be a  Lipschitz continuous $\lambda$-convex function.
	
	Then for  the composition $u\circ f \in W^{1,2} _{loc} (\Omega)$ the distributional Laplacian $\Delta (f\circ u)$ is a signed locally finite measure which satisfies
	$$\Delta (f\circ u) \geq \lambda \cdot  e^2 _u \;,$$
where $e^2_u \in \mathcal L^1 (\Omega)$ is the energy density of $u$.
\end{thm}

 Any Sobolev function with vanishing energy density has a constant representative.    As an immediate  consequence we obtain:

 \begin{cor}
	Let $\Omega$ be a domain in  a Euclidean space $\R^n$ and let $u:\Omega \to X$ be a harmonic map into a CAT($\kappa $) space $X$.  Let $f:X\to \R$ be  Lipschitz and  $1$-convex.
	If the composition $f\circ u$ is constant then $u$ itself is a constant map.
\end{cor}

\subsection{Comments}
The proof of Theorem \ref{thm: fuglede}  does not use any specific property of the domain $\Omega$ and applies without changes to any domain in a Riemannian manifold and
to domains in admissible Riemannian polyhedra in the sense of \cite{Fuglede}, \cite{Mese-general}, \cite{Mese-Breiner-regular}.

The  proof of Theorem \ref{thm: fuglede}  does not use the upper curvature bound assumption in an essential way either. It only relies on the
existence and the contracting behaviour of gradient flows of semi-convex functions, which is valid in much greater generality. For instance these features are true in spaces with lower curvature bounds,
 \cite{AKP}, \cite{Petruninsemi} and in some spaces of probability measures, \cite{AGS}, \cite{Ohta-2}.

The importance  of the result for spaces of curvature bounded above lies in a great variety of semi-convex functions.  For instance, for every point $x$ in a CAT($\kappa$) space $X$ the  function $f(y)=d^2(x,y)$ is $1$-convex if $\kappa \leq 0$.  If $\kappa >0$ then the function $f$  is $\epsilon$-convex on the closed ball
$\bar B=B_r(\kappa)$ for any $r <\frac {\pi} {2\sqrt \kappa}$ and some $\epsilon =\epsilon (r,\kappa)>0$.   This in conjunction with Theorem \ref{thm: confchange}   leads to the proofs of
Theorems \ref{thm: main}, \ref{thm: nonpos}.

Moreover,  on any CAT(0) space $X$ the distance function $d:X\times X\to \R$ is convex. Similarly, for any closed ball $B$ of radius less than $\frac {\pi} {2\sqrt \kappa}$ in any CAT($\kappa$) space
there exists a convex function $\psi :B\times B\to \R$, comparable (up to a bounded factor) with the distance $d(x,y)$, \cite{Kendall}, \cite{Yokota}.   This directly  implies  the uniqueness of solutions
of the Dirichlet problem and  the continuous dependence of harmonic maps on their traces.
The existence of the harmonic maps with prescribed trace and their regularity involve some finer  arguments but are heavily based on
Theorem \ref{thm: fuglede} for the functions  $d$ and $\psi$, respectively,
\cite{KS}, \cite{Serbinowski}.

The proof of Theorem \ref{thm: confchange} provides also a local statement in the case $\kappa -4\lambda >0$ and $\lambda \leq 0$, see Theorem \ref{thm:locdefo}.
Moreover, as the proof of Theorem \ref{thm: main}  shows, one can localize
the statement by writing the formulas in Theorem \ref{thm: confchange} using  only local bounds and local semi-convexity  of $f$.

\subsection{Structure of the paper}
In the preliminaries we recall basics of Sobolev maps and variation of length under gradient flows of semi-convex functions.
This variation is applied in Section \ref{sec:fuglede} to obtain a proof of Theorem \ref{thm: fuglede}.
In Section \ref{sec:prepar}, we recall some structural results about spaces with upper curvature bounds, minimal discs  and conformal changes used in Section \ref{sec:main} to prove the main results of the paper.

\subsection{Acknowledgments}  We would like to thank Grigori Avramidi for posing  a question which has lead to Theorem \ref{thm: main}.

Both authors were partially supported by DFG grant  SPP 2026.

\section{Preliminaries}
%\subsection{Notation}
\subsection{Notations and spaces with upper curvature bounds}
We refer the reader to \cite{BBI01}, \cite{BH} and \cite{AKP} for basics on metric geometry and
CAT($\kappa$) spaces. Here we just agree on notation, some finer properties will be discussed in Subsection \ref{subsec: facts}.

In this paper all CAT($\kappa$) spaces will be complete length spaces by definition.
By $d$ we will denote distances in metric spaces.  We will let $D\subset\R^2$ be the open Euclidean unit disc, $\bar D\subset\R^2$ the closed Euclidean unit disc and
$S^1=\D\bar D\subset\R^2$ the unit circle.

 For a  Lipschitz function $f:X\to \R$ on a metric space $X$ we denote by $|\nabla ^- _p f| \in [0,\infty )$
the {\em descending slope} of $f$ at $p \in X$ defined by
$$|\nabla^-_p f|=\max\{0,\limsup_{x\to p}\frac{f(p)-f(x)}{d(p,x)}\}.$$

\subsection{Semi-convex functions and their gradient flows}

Let $X$ be a CAT($\kappa$) space.
For any  Lipschitz continuous $\lambda$-convex  function $f:X\to \R$ there exists the locally Lipschitz continuous  \emph{gradient flow} $\Phi:[0,\infty) \times X \to X$  of $f$,
such that for any $x$ the flow line $t\to \Phi _t(x)$ is the gradient curve of the function $f$ starting at $x$.

As a reference  one  can use  \cite{Ohta} or \cite{Ly-open}, see also \cite{Petruninsemi},
\cite{Mayer} and \cite{AGS} for a general theory of gradient flows in metric spaces.

From all properties of gradient flows
we will only need the following distance estimate on the change of length under the gradient flow.
In \cite[Lemma 2.2.1, Lemma 2.1.4]{Petruninsemi} it is proven for Alexandrov spaces but the proof relies only on the first variation formula and is identical in our setting of
CAT($\kappa$) spaces.

\begin{cor} \label{cor: distesti}
	Let $X$ be  CAT($\kappa$) and let $f$ a  Lipschitz continuous $\la$-convex function on $X$ with   gradient flow  $\Phi$. Let
	$\gamma:[a,b] \to X$ be an absolutely continuous curve   and let $\rho :[a,b] \to [0,\infty)$ be  Lipschitz.

	Then $\eta (s):=   \Phi _{\rho (s)} (\gamma (s))$ is an  absolutely continuous  curve and   for almost all $s\in  [a,b]$
	its velocity is bounded by
		$$|\eta '(s)|^2 \leq e^{-2\lambda \cdot \rho (s)}
		(|\gamma '(s)|^2 -    2(f\circ \gamma)' (s) \cdot \rho ' (s)   + |\nabla ^- _{\gamma (s)}f |^2  \cdot (\rho'(s))^2).$$

\end{cor}

\subsection{Sobolev maps and energy}

By now there exists a well established theory of Sobolev maps with  values in metric spaces, \cite{HKST15}.
We will follow   \cite{LWplateau} and restrict our revision to the special case needed in this paper.

Let $X$ be a complete metric space. Let $\Omega \subset \R^n$ be a Lipschitz domain and denote by $L^2(\Omega,X)$ the set of measurable and essentially
separably valued maps $u:\Omega\to X$ such that for some and thus every $x\in X$ the function $u_x(z):=d(x,u(z))$
belongs to  $L^2(\Omega)$.

\begin{defn}
 A map $u\in L^2(\Omega,X)$ belongs to the Sobolev space $W^{1,2}(\Omega,X)$ if there exists  $h\in L^2(\Omega)$ such that for every $x\in X$
 the composition $u_x$
 is contained in the classical Sobolev space $W^{1,2}(\Omega)$ and its weak gradient satisfies $|\nabla u_x|\leq h$ almost everywhere on $\Omega$.
\end{defn}

Each Sobolev map $u$ has an associated {\em trace} $\tr(u)\in L^2(\D\Omega ,X)$, see \cite{KS}. If $u$ extends continuously to a map $\hat u$ on $\bar \Omega$, then
$\tr(u)$ is represented by the restriction $\hat u|_{\D \bar \Omega}$.

There are several natural definitions of energy for Sobolev maps, see \cite[Section 4]{LWplateau}.
We will only use the {\em Korevaar--Schoen energy}. It can be defined in many different ways, for instance, using the approximate metric differentials \cite[Proposition 4.6]{LWplateau}.  The expression we are going to use is the following one.

 Any  map $u\in W^{1,2} (\Omega, X)$ has a representative, also denoted by $u$, which is absolutely continuous on \emph{almost all curves} in $\Omega$, \cite{HKST15}.  Then, for any vector $v\in \R^n$ the restriction of $u$ to almost any segment parallel to $v$ is absolutely continuous, hence has a well defined finite velocity at almost all times.
 Thus the function $m_u (x,v)$ which measures the velocity of the curve
 $t\to u(x+tv)$ at $t=0$ is well-defined almost everywhere on $\Omega \times \R^n$.  We mention that, for almost all $x\in \Omega$, the function $v\to m_u (x,v)$ is a semi-norm on $\R^n$, the \emph{approximate metric differential} of $u$ at $x$, which in fact is Euclidean, if $X$ is CAT($\kappa$), \cite{KS}, \cite[Section 11]{LWplateau}.

  The \emph{energy} density  of $u$ is defined as
$$e^2 _u (z) = \frac{1}{\om_{n}}\int_{S^{n-1}}  |m_u (x,v) |^2 \; dv \,,$$
where $\om_n$ denotes the Lebesgue measure of the unit ball in $\R^n$. The  {\em Korevaar--Schoen energy} of $u$ is given by
$$E^2(u):=\int\limits_{\Omega }  e^2 _u (z) \; dz \;.$$
The map $u\in W^{1,2} (\Omega, X)$ is called \emph{harmonic} if
 for all $w\in W^{1,2} (\Omega, X)$ with the same trace as $u$ one has
 $$E^2 (u)  \leq E^2(w) \;.$$

 If $X$ is CAT(0) or a CAT($\kappa$) space of sufficiently small diameter, then any harmonic map
 is locally Lipschitz continuous and  uniquely determined by a prescribed  trace \cite{KS}, \cite{Serbinowski}, \cite{Fuglede-reg}, \cite{Mese-Breiner-regular}.

\section{First variation of energy} \label{sec:fuglede}

For Sobolev maps in CAT($\kappa$) spaces we have the following analog of the classical first variation formula.

\begin{lem}\label{lem:varinequ}
Let $X$ be a CAT($\kappa$) space and $f$ a  Lipschitz continuous $\lambda$-convex function on $X$ with gradient flow
$\Phi:[0,\infty)\times X\rightarrow X$.

Let $u\in W^{1,2} (\Omega, X)$ be given.
 For any Lipschitz continuous test function $\rho: \Omega \to[0,\infty)$
with compact support in $\Omega$, define a variation $u_t$ of $u$ by
\[u_t(x)=\Phi( t\rho(x) ,u(x)).\]
Then $u_t \in W^{1,2} (\Omega, X)$ and  the following inequality holds
\begin{equation}\label{eq:var}
 \frac{d}{dt}^+\Bigr|_{t=0}E^2(u_t)\leq -2  \int\limits_{\Omega} \left(\lambda \cdot  e^2_u(x)\cdot \rho(x)  +
 \langle\nabla_x (f\circ u),\nabla_x\rho\rangle \right)\; dx
\end{equation}
where the left-hand side is the upper Dini derivative.
\end{lem}

\begin{proof}
	We fix $t>0$. As a composition of a Sobolev and a Lipschitz map, the map $u_t$ 	is contained in $W^{1,2} (\Omega, X)$.

	Consider  the curves $\gamma (s) = u(x+sv)$ and $\eta (s)=
	\Phi (t\cdot \rho (x+s v), \gamma (s))$.

	By definition, for almost all $x,v \in \Omega \times S^{n-1}$,  the velocities of $\gamma$ respectively $\eta$  at $s=0$ are exactly
	$m_u(x,v)$ and $m_{u_t} (x,v)$, the values of the corresponding approximate metric differentials.  Applying    Corollary \ref{cor: distesti} we get the estimate, valid at all such $x,v$:

	$$m_{u_t}^2  (x,v)  \leq e^{-2\lambda \cdot t\cdot \rho (x)}
	(m_u ^2(x,v) -    2t\left\langle\nabla _x (f\circ u) ,v\right\rangle  \cdot \left\langle\nabla _x \rho, v\right\rangle     + t^2 |\nabla _{u(x)}f |^2  \cdot \left\langle\nabla _x \rho ,v\right\rangle^2)$$
Averaging over $S^{n-1}$  and using  the equality
	$$\frac 1 {\omega _n} \int _{S^{n-1}}  \langle w_1,v \rangle \cdot \langle w_2, v \rangle \;dv =  \langle w_1,w_2 \rangle  \;,$$
	 valid for all $w_1,w_2\in \R^n$   ,  we obtain the estimate of the  energy densities,
valid pointwise almost everywhere on $\Omega$:
$$e^2_{u_t} (x) \leq 	e^{-2\lambda \cdot t\cdot \rho (x)}(e^2_u (x) -
	2t  \left\langle\nabla _x \rho , \nabla _x (f\circ u)\right\rangle +t^2 |\nabla _{u(x) }f|^2 \cdot |\nabla _x \rho|^2 ) \leq $$
$$\leq	(1-2\lambda \rho (x) \cdot t) (e^2_u(x) -	2t  \left\langle\nabla _x \rho , \nabla _x (f\circ u)\right\rangle   )  +Ct^2 \;,$$
for some constant $C$ depending on $\lambda$, $\rho$, $f$ and $u$.

 The claim follows now directly by integration over  $\Omega$.
\end{proof}

Now we can easily derive:
\begin{proof}[Proof of Theorem \ref{thm: fuglede}]

	Appyling Lemma \ref{lem:varinequ} and the definition of harmonicity, we see that the right hand side of (\ref{eq:var}) must be non-negative for
	any non-negative  Lipschitz continuous function $\rho$ with compact support in $\Omega$.
	Thus the distributional Laplacian $\Delta (f\circ u)$ satisfies
	$$\Delta (f\circ u) (\rho)  = - \int _{\Omega} \langle \nabla _x (f\circ u), \nabla _x \rho \rangle  \; dx  \geq \lambda \int _{\Omega}  e^2_u \cdot \rho \; .$$
	By the representation theorem of Riesz in distirbution theory, this is sufficient to draw the conclusion.
\end{proof}

\section{Preparations} \label{sec:prepar}
\subsection{Length spaces and their conformal changes}  \label{subsec: confchnage}
The length of a rectifiable curve $\gamma$ in a metric space $X$ is denoted by $\ell (\gamma)$.   A metric space $X$ is a \emph{length space} if the distance between any two points
is equal to the greatest lower bound for lengths of curves connecting the respective points.
A curve $c:[a,b]\to X$ will  be called {\em geodesic} if  it is an isometric embedding.
The space $X$ itself will be called {\em geodesic} if any two points in $X$
are joined by a geodesic.

We refer to \cite{LS} for more details on what follows here. Let $X$ be a length space and  $f:X\to (0,\infty)$ be a  continuous function.

We define the $f$-length of a rectifiable curve $\gamma :[a,b]\to X$  by
\begin{equation} \label{eq:length}
\ell_f(\gamma) = \int_a^b f(\gamma(t)) \cdot |\dot\gamma(t)| \,dt \,,
\end{equation}
where $|\dot\gamma(t)|$ denotes the velocity of the curve $\gamma$ at  time $t$.   The conformally changed metric $d_f$ on the space $X$ is defined by
\begin{equation} \label{eq:dist}
d_f (x,y)=  \inf _{\gamma} \left\{ \ell _f(\gamma) \; ; \;  \gamma \;  \text{Lipschitz curve from  } x \; \text{ to } \;  y\right\} \;.
\end{equation}
 The space   $f\sd X:=(X,d_f)$ is a length space called  the \emph{metric space conformally   equivalent to $X$ with conformal factor $f$}.

The identity map  $\id_f:X\to f\sd X$ is  a  locally bilipschitz homeomorphism.
If $f$  is bounded from below  by a positive constant  and $X$ is complete then $f\cdot X$ is complete as well.

We will need the following observation:

\begin{lem}  \label{lem: distfunc}
	Let $X$ be a length space and assume  $X=B_r(x)$ for some    $x\in X$ and $r>0$.
	Let $\xi :[0,r) \to (0,\infty)$ be   continuous  and consider the function
	$f(y):=\xi  (d(x,y))$ on $X$.  Then, in the conformally changed space $f\cdot X$, the distance function to  $x$ can be computed as:
	$$d_f (x,y)  = \int _0 ^{d(x,y)} \xi (t) \; dt \;.$$
	
\end{lem}

\begin{proof}
	
	Consider concentric metric spheres $S_s(x) $ of radii $s \leq d(x,y)$ around $x$.  Then for any
	$0 \leq  s_0<s_1 \leq d(x,y) $  and any point $  z\in S_{s_1}(x)$ we can estimate the $d_f$-distance from $z$ to $S_s(X)$  as
	$$|s_1-s_0| \cdot \min _{s_0\leq s \leq s_1}  \xi (s)  \leq       d_f (z, S_s(x)) \leq |s_1-s_0| \cdot \max _{s_0\leq s \leq s_1}  \xi (s) \,.$$
	The proof of the lemma follows by writing the integral $\int _0 ^{d(x,y)} \xi (t) \; dt $ as a limit of Riemann sums as on the right and left hand sides of the above inequality.
	
\end{proof}

\subsection{Local-to-global in CAT($\ka$) spaces}\label{subsec: facts}
The basic local-to-global theorem about CAT($\ka$) spaces is the theorem of Cartan--Hadamard, saying that a complete length space, which is locally CAT($\ka$) with $\kappa \leq 0$, is a CAT($\ka$) space if and only if it is simply connected.

In order to describe  related local-to-global statements for all
$\kappa$, we recall that the injectivity radius  of a local CAT($\ka$) space $X$
 is the supremum of all $r>0$, such that any pair of points $x$ and $y$ in $X$ at distance less than $r$ are connected by a unique geodesic and that this geodesic depends continuously on the endpoints.

 Combining \cite[6.10]{Ballmann} and  \cite[8.11.3]{AKP} we obtain the following.
\begin{lem}  \label{lem:loc-glob}
	Let $X$ be a complete length space which is locally CAT($\ka$).
Let $r>0$ be such that $r<\frac {\pi} {2\sqrt \kappa}$ for $\kappa >0$.

	If the injectivity radius of $X$ is larger than $r$  then  any  ball $\bar B_{ r } (x)$ is a  convex CAT($\ka$) subset of $X$.
\end{lem}

 The next local-to-global result is also well-known.
\begin{lem}\label{lem:shorthom}
	Let $X$ be a complete length space which is locally CAT($\ka$).
		Let $\Lambda >0$ be such that $\Lambda \leq \frac {\pi} {\sqrt \kappa}$ if $\kappa >0$.

	Assume that for any closed curve $\Gamma$ of length less than $2\Lambda$
	there exists a homotopy $\Gamma _t , t\in [0,1]$  from
	$\Gamma =\Gamma _0$ to a constant curve $\Gamma _1$ such that  the length of $\Gamma _t$ is less than $2\Lambda $ for all $t$.

	Then any closed  ball of radius  less than $\frac{\Lambda}{2}$ in $X$ is convex and CAT($\ka$).
\end{lem}

\begin{proof}
 Let $\Ga$ be a closed curve of length less than $2\La$. Then our assumptions allow to apply \cite[8.13.4]{AKP}, to conclude that $\Ga$ is majorized by a CAT($\ka$) space.
 It follows that the injectivity radius of $X$ is at least $\frac{\Lambda}{2}$. Hence Lemma \ref{lem:loc-glob} completes the proof.
\end{proof}

Corollary \ref{cor:global} below will
 slightly strengthen the next result.

\begin{prop} \label{prop:global}
	Let $Z$ be a compact geodesic space homeomorphic to $\bar D$.
	If $Z$ is locally CAT(1) and has Hausdorff area  $\mathcal H^2 (Z)$  less than $2\pi$, then $Z$ is CAT(1).
\end{prop}

\begin{proof}
	Otherwise $Z$  contains  an isometric embedding $\Gamma$ of a round  circle $S^1_{2l}$ of length $2l <2\pi$ in $Z$, \cite[6.9]{Ballmann}.  The closed Jordan domain $Z_1$ cut out of $Z$ by $\Gamma$ is convex, hence locally CAT(1) and its Hausdorff area is also less than $2\pi$.

	Doubling $Z_1$, thus gluing two copies of it along $\Gamma$,  we obtain
	a  space homeomorphic to the $2$-dimensional sphere, which has Hausdorff area less than $4\pi$ and which is  locally CAT(1), by Reshetnyak's gluing theorem, \cite[Theorem 8.9.1]{AKP}.

	But this contradicts the Gauss-Bonnet formula, \cite[(8.15)]{Reshetnyak-GeomIV}.
\end{proof}

As a consequence we deduce the following analog of \cite[Proposition 12.1]{LWcurv} for $\kappa \neq 0$.

\begin{cor}\label{cor:global}
	Let $\hat Z$ be a length space homeomorphic to the open disc $D$. For $\kappa \in \R$ let   $\hat Z$ be locally CAT($\ka$) and assume that
	for $\kappa >0$ the area $\mathcal H^2(\hat Z)$   is  at most $\frac {2\pi }{\kappa}$. Then the completion $Z$ of $\hat Z$  is CAT($\ka$).
\end{cor}

\begin{proof}
	We exhaust the space $\hat Z$ by compact closed discs  $Z_n$ with boundary being a geodesic polygon as in the proof  of  \cite[Proposition 12.1]{LWcurv}.	As in   \cite[Section 11.2]{LWcurv}, we readily see that
	these subsets  $Z_n$ are locally CAT($\kappa$ ) in their intrinsic metrics.

	Moreover, a  limiting argument as in   \cite[Proposition 12.1]{LWcurv} shows that it  suffices to verify that  the closed discs  $Z_n$ are globally CAT($\ka$).

	For $\kappa \leq 0$, this statement follows directly by the theorem of Cartan--Hadamard.  For $\kappa >0$, we may rescale the space and assume $\kappa =1$.
	Any open non-empty subset of $Z_0$ has positive $\mathcal H^2$-area, \cite[Theorem 1.2]{LN1}, thus $\mathcal H^2(Z_n) <2\pi$, for any $n$.  The global CAT(1)  property of $Z_n$ is exactly
	 Proposition \ref{prop:global}.
\end{proof}

\subsection{Recognizing CAT($\ka$) spaces}
For us it will be important that CAT($\ka$) spaces can be recognized
without referring to geodesic triangles.
By a \emph{Jordan curve} in a metric space $X$ we denote a subset
homeomorphic to a circle.

We say that a metric space $Y$ majorizes a  rectifiable Jordan curve $\Gamma$ in a metric space $X$
if there exists a $1$-Lipschitz map $P:Y\to X$ which sends a Jordan curve $\Gamma'  \subset Y$ bijectively in an arc length preserving way onto $\Gamma$.
The following is proved in \cite[Proposition 2]{LS} for CAT(0) spaces.  Along the same lines we deduce:

\begin{prop}  \label{prop:major}
 Let $\kappa \in \R$ and $\Lambda >0$ be such that $\Lambda \leq \frac {\pi} {\sqrt k}$ if $\kappa >0$. Let $X$ be a complete length metric space.

 If any Jordan curve $\Gamma$ in $X$ of length $< 2\Lambda$
is majorized by some CAT($\ka$) space $Y_{\Gamma}$, then any closed  ball $B$  in $X$ of any  radius $r<\frac {\Lambda} 2$ is convex in $X$ and CAT($\kappa$) .

Moreover, if $\kappa >0$ and $\Lambda  =\frac {\pi} {\sqrt k}$ then $X$ is CAT($\ka$).
\end{prop}

\begin{proof}
 The argument in  \cite[Proposition 2]{LS} shows that any pair of points in
 $X$ at distance less than $\Lambda$ is connected by a unique geodesic in $X$. Moreover, such geodesics depend continuously on their endpoints.

  As in \cite[Proposition 2]{LS} the assumption implies that any triangle in $X$ of perimeter less than $2\Lambda$ is not thicker than its comparison triangle in the constant curvature surface.

  For $\kappa >0$ and $\Lambda  =\frac {\pi} {\sqrt k}$ this implies by definition, that  $X$ is CAT($\ka$).

  For general $\Lambda$, the condition implies that $X$ is locally CAT($\ka$) and the statement follows from Lemma \ref{lem:loc-glob}.
 
\end{proof}

\subsection{Surfaces}

In the case of flat domains the curvature of conformally changed metrics has been investigated
in detail by Yuri Reshetnyak, see \cite{Reshetnyak-GeomIV}  and the references therein.  In this case it is even possible
to relax the continuity and positivity assumptions on conformal factors.

 We say that a function $f:U\to [0,\infty)$
on a domain $U\subset \R^2$  is    $\kappa$-{\em log-subharmonic}, if
$f$  is upper semi-continuous, contained in $L^1_{loc}$ and satisfies
weakly
$$\Delta \log f+\frac{\ka}{2} f^2\geq 0 \;.$$
For a $\ka$-log-subharmonic function $f$ one can use  formulas \eqref{eq:length} and \eqref{eq:dist}
to define the conformally changed metric on $U$. Indeed, we have the following result due to Reshetnyak,
see Theorem 7.1.1 in \cite{Reshetnyak-GeomIV},  see  also \cite[Theorem 6.1]{Mese-curvature} and  \cite[Theorem 8.1, Section 17]{LWcurv}.

\begin{thm}\label{thm:resh}
 Let $U\subset\R^2$ be a domain and $f$ a $\ka$-log-subharmonic function on $U$.
 Then $f\sd U$ is locally CAT($\ka$)  and $\id_f:U\to f\sd U$
 is a homeomorphism.
\end{thm}

The next computational lemma will provide control on double conformal changes.

\begin{lem}\label{lem:prod_of_ka} Let $c,C, \kappa, \lambda \in \R$.
Let $U\subset\R^2$ be a domain and $\varphi$ a $\ka$-log-subharmonic function on $U$.
Suppose that $\psi:U\to [c,C]$ is a continuous function which safisfies $\Delta\psi\geq\mu\cdot\varphi^2$ weakly.
Then the product $e^\psi\cdot\varphi$ is $\bar\ka$-log-subharmonic, with

\begin{itemize}
 \item $\bar\ka=e^{-2C}\cdot(\ka-2\mu)$ if $\ka-2\mu\leq 0$;
 \item  $\bar\ka=e^{-2c}\cdot(\ka-2\mu)$ if $\ka-2\mu\geq 0$.
\end{itemize}

\end{lem}

\begin{proof}
  Suppose $\ka-2\mu\leq 0$. Then $-(\kappa -2\mu)  \geq -\bar \kappa \cdot (e^\psi )^2$ and the claim follows from
 \begin{align*}
  {\Laplace\log(e^\psi \varphi)} =  {\Laplace\psi} +  {\Laplace\log\varphi}
  \geq \mu \cdot \varphi ^2 - \frac {\kappa} 2 \varphi ^2 = -\frac 1 2 (\kappa- 2\mu) \varphi ^2
   \end{align*}
 
 The case  $\kappa- 2\mu \geq 0$ is analogous.
\end{proof}

Combining Theorem \ref{thm:resh} and Lemma \ref{lem:prod_of_ka} leads to:

\begin{lem} \label{lem:cat}
Let $c,C, \kappa, \mu \in \R$ and set
\begin{itemize}
 \item $\bar\ka=e^{-2C}\cdot(\ka-2\mu)$ if $\ka-2\mu\leq 0$;
 \item  $\bar\ka=e^{-2c}\cdot(\ka-2\mu)$ if $\ka-2\mu\geq 0$.
 \end{itemize}
Suppose that $\varphi$ is a $\ka$-log-subharmonic function on $D$
and let $\varphi \sd D$ be the conformally changed disc. Let  $Z$ denote the completion of $\varphi  \sd D$. If $\bar\ka>0$, assume in addition  $\mathcal H^2(\varphi \sd D) \leq e^{-2C}\cdot\frac{2\pi}{\bar\ka}$.
Finally, let $\psi:Z \to[c,C]$  be a continuous function on $Z$ such that the restriction of $\psi$ to $D$ satisfies
$\Delta\psi\geq\mu\cdot\varphi^2$ weakly. Then $e^\psi\sd Z$ is CAT($\bar\kappa$).
\end{lem}

\begin{proof}
The proof of Lemma 4 of \cite{LS} implies
$e^\psi\sd Z$ is the completion of the length space
$e^\psi\sd (\varphi \sd D)$. Moreover, it shows that
$(e^\psi\sd \varphi)\sd D$ is isometric to $e^\psi\sd(\varphi\sd D)$.
Hence  Lemma \ref{lem:prod_of_ka} together with Theorem \ref{thm:resh} imply that $(e^\psi\sd \varphi)\sd D$ is locally
CAT($\bar\kappa$).

The claim then  follows from Corollary \ref{cor:global}.
\end{proof}

\subsection{Minimal discs}
A general solution of the classical Plateau's problem has been provided in \cite{LWplateau} for proper metric spaces and in \cite{Guo}  for complete CAT(0) spaces. We need to discuss an appropriate extension to non-proper CAT($\ka$) spaces with $\kappa >0$.

\begin{lem}\label{lem:PP}
	Let $X$ be a CAT($\ka$) 	space and let  $\Gamma$ be a  Jordan curve in $X$ of finite length $l$. If $\kappa  >0$ assume in addition, that $l < \frac {2\pi} {\sqrt \kappa}$. Then there exists a closed ball
	 $B=\bar B_r(x)$  which contains $\Gamma$. Moreover, if $\kappa >0$ we can choose $r< \frac {\pi} {2 \sqrt \kappa}$.	%

The space 	$\Lambda(\Gamma, B)$ of all maps $v\in W^{1,2}(D,B)$ such that $\tr(v)$ is a weakly monotone parametrization of $\Gamma$ is non-empty and contains a map $u_0$ of smallest energy 
  $E^2(u_0) < \frac 1{\pi} \cdot l^2$ in 	$\Lambda(\Gamma,B)$.

 Moreover, any such map $u_0$ has a unique representative which extends continuously to $\bar D$.
\end{lem}

\begin{proof}
	Without loss of generality, we may assume $\kappa=1$.
	The existence of the required ball $B$ is a consequence of Reshetnyak's majorization theorem \cite[8.12.4]{AKP}.
	Moreover, putting together the majorization theorem, the isoperimetric inequality in the hemisphere and the fact that convex subdomains of the hemisphere allow for conformal parametrizations, 
	we conclude that there exists  a map 
	$u\in  \Lambda (\Gamma, B)$ with $E^2(u) <  \frac 1 {\pi} \cdot l^2$. In particular,  $\Lambda (\Gamma, B)$ is non-empty.

	If $B$ is compact, the existence of an energy minimizer in $\Lambda (\Gamma, B)$ is proved in \cite{LWplateau}. The same classical  argument, extended in \cite{LWplateau} to proper metric spaces also works in the
	present  non-compact case as follows.

  As in the classical case, we  can precompose with Moebius maps and restrict to the subspace $\Lambda _0 (\Gamma)$ of all maps in $\Lambda (\Gamma)$ whose
  trace sends three fixed points in $S^1$ to three prescribed points in $\Gamma$, \cite[Section 7]{LWplateau}.
	Take an energy minimizing sequence $u_n$ in $\Lambda _0 (\Gamma)$.
	For any $u_n$, we find a unique harmonic map $v_n \in \Lambda_0 (\Gamma)$ with the same trace as $u_n$, \cite{Serbinowski}, \cite[Theorem 4]{Fuglede-reg}, since $r< \frac {\pi} {2}$. In particular, $E(v_n)\leq E(u_n)$ and $v_n$ is an energy-minimizing sequence as well.

	By the lemma of Courant--Lebesgue, the traces $tr(v_n)$ are uniformly continuous, \cite[Section 7]{LWplateau}. By \cite{Serbinowski}, \cite[Theorem 2]{Fuglede-reg}, the maps $v_n$ have unique representatives,
	which continuously extends to $\bar D$. Moreover,  these representatives  converge uniformly, once their traces converge uniformly.  Thus, using the semi-continuity of energy, we obtain
	an energy minimizer in $\Lambda_0 (\Gamma)$ by taking a uniform limit of a subsequence of the maps $v_n$.
\end{proof}

  We call a continuous map $u_0:\bar D\to B$ provided by the above result
   a \emph{minimal filling}  of $\Gamma$ in $B$ and are going to summarize its properties:

\begin{thm}\label{thm:plateau}
 Let $\Gamma$ be a   Jordan curve of length $l$ in a CAT($\ka$) space
 $X$, with $l <\frac {2 \pi} {\sqrt {\kappa}}$ if $\kappa >0$.  As in Lemma \ref{lem:PP}, let $B=B_r(x)$ be a closed ball  which contains $\Gamma$   and let $u :\bar D\to B$ be a minimal filling of   $\Gamma$.

 Then the following hold true:

 \begin{enumerate}
  \item  $u$ is harmonic  and $E^2 (u) <  \frac 1 {\pi} l^2$.
  \item  There exists a function $\varphi \in L^2(D)$, the  {\em conformal factor} of $u$, such that
  the approximate metric differential satisfies $m_u (z,v) =\varphi (z) \cdot \|v\|$   for almost all $z\in D$ and all $v\in \R^2$.
  \item  The conformal factor $\varphi$ can be chosen to be  $\kappa$-log-subharmonic.
  \item The completion $Z$ of $\varphi \sd D$ is a CAT($\kappa$) space. 
  \item  The space $Z$ is homeomorphic to $\bar D$, the  map $u:\varphi \sd D \to X$ is $1$-Lipschitz and  extends to a majorization  $v:Z\to X$  of  $\Gamma$. 
 \end{enumerate}
\end{thm}

\begin{proof}
 (1) follows by definition	of a minimal filling and Lemma \ref{lem:PP}. 

 (2) is verified in \cite[Theorem 11.3]{LWplateau}

 (3) is verified  in \cite{Mese-curvature}.

 In order to verify (4), we use (3) and Theorem \ref{thm:resh} to see that $\varphi \sd D$ is locally CAT($\kappa$).   The area of $\varphi \sd D$ can be computed as  
 $\mathcal H^2 (\varphi \sd D)  =\int _D \varphi ^2 =\frac{1}{2} E^2 (u)< \frac 1 {2\pi} l^2 $.
In particular,  if $\kappa >0$, we have  $\mathcal H^2(\varphi \cdot D) <\frac {2\pi} {\kappa} $.
 Thus, (4) is a consequence of Corollary \ref{cor:global}.

 Finally, (5) is verified in \cite[Theorem 9]{LS} for $\kappa =0$, hence also for $\kappa \leq 0$. The proof applies without changes to the case $\kappa >0$.
\end{proof}

\section{Main result} \label{sec:main}
\subsection{Local control of curvature under conformal changes}
Along the lines of \cite{LS}, we are going to prove the following local version  of Theorem \ref{thm: confchange}.

\begin{thm}\label{thm:locdefo}
		For $c,C, \kappa, \lambda \in \R$ there exists some
		$\rho_0 =\rho_0 (c,C,\kappa, \lambda) >0$ with the following property.

		Let $X$ be a CAT($\kappa$) space and let $f:X\to [c,C ]$ be a  Lipschitz continuous $\lambda$-convex function.   Further, let  $Y= e^{f} \cdot X$ denote the conformally equivalent space.  Then any closed
		ball of radius at most $\rho_0$ in $Y$ is CAT($\bar\kappa$), where
	\begin{itemize}
		\item
		$\bar \kappa 	= e^{-2C} \cdot  (\ka -4\lambda )$ if  $\ka  -4\lambda  \leq 0$;
		
		\item
		$\bar \kappa 	= e^{-2c} \cdot  (\ka -4\lambda )$ if $\ka  -4\lambda  \geq 0$.
		
	\end{itemize}
	\end{thm}

\begin{proof}

Set $Y= e^f\cdot X$.  Since $f$ is bounded, the identity map from $Y$ to $X$ is a bilipschitz homeomorphism. Thus any complete subset of $X$ is also complete in $Y$.
Rescaling $Y$ by the factor $e^{-c}$, thus  subtracting the constant $c$ from $f$ we may assume that $c=0$.

Choose a positive constant $\La < e^{-C} \cdot  \frac{2\pi}{\sqrt{\kappa}}$.

We claim that any Jordan curve $\Gamma$ in $Y$ of length $< \La$   is majorized by a CAT(${\bar\ka}$) space.  Fix $\Gamma$ and denote by $\hat \Gamma$ the curve $\Gamma$ considered in $X$.
Since  the identity map $Y\to X$ is $1$-Lipschitz, the length of $\hat \Gamma$ in $X$ is at most $\Lambda < \frac{2\pi}{\sqrt{\kappa}}$.

By Lemma \ref{lem:PP}, we obtain a minimal filling  $u$ of  $\hat \Ga$ in $X$, whose properties are described in Theorem \ref{thm:plateau}.
 Denote by $\varphi$ the conformal factor of $u$.  By Theorem \ref{thm:plateau},  $u:\varphi\sd D\to X$
extends to a majorization $v:Z\to X$, where the completion  $Z$ of $\varphi\sd D$ is  CAT($\ka$).    Moreover,  the area of $Z$ is less than
$\frac 1{2\pi } \Lambda ^2 < e^{-2C}\cdot \frac {2\pi} {\kappa}$.

Due to  Theorem \ref{thm: fuglede} and the conformality of $u$, the composition $f\circ u$ fulfills $\Delta(f\circ u)\geq 2\la\cdot \varphi^2$ weakly. Hence, Lemma \ref{lem:cat}  ensures that $e^{(f\circ v)}\sd Z$ is CAT($\bar\ka$).    Moreover, the majorization $v:Z\to X$ of $\hat \Gamma$ defines a majorization $v: e^{(f\circ v)}\sd Z \to e^f \cdot X=Y$ of $\Gamma$.

Thus, any Jordan curve $\Gamma$ of length less than $\Lambda$ in $Y$ is majorized by a CAT($\bar \kappa$) space. We finish the proof by setting $\rho_0=\frac{\La}{4}$ and 
 applying  Lemma \ref{lem:shorthom}.

\end{proof}

\subsection{Global versions}

 Now we can turn to the main theorems.

\begin{proof} [Proof of Theorem \ref{thm: confchange}]

	Due to Theorem \ref{thm:locdefo}, the space $Y$ is a complete length space, which is locally CAT($\bar\kappa$).  It remains to globalize the statement.

	Assume first that $\lambda >0$ and consider the gradient flow $\Phi _t$
	of the function $f$ on the space $X$.  Let $\Gamma$ denote a rectifiable closed curve in $Y$.  Considering $\Gamma$ as a curve in $X$,
	we apply  the gradient flow $\Phi _t$ to $\Gamma$ and obtain closed curves $\Gamma_t$ in $X$.  The value of $f$ (and hence of $e^f$) does not increase along flow lines of $\Phi$.
	Since $\Phi _t$ contracts length in $X$,  at least by a factor of $e^{-\lambda t}$, Lemma \ref{lem:varinequ}, we deduce the following two consequences.  Firstly, the $e^f$-length of $\Gamma _t$  (thus the length of $\Gamma _t$ in $Y$) is
	non-increasing in $t$. Secondly, for any $\Gamma$ as above, any $\epsilon>0$  and any  sufficiently large $t$,  the length of $\Gamma _t$ in $Y$ is less than $\epsilon$.

	  Taking $\epsilon$ to be smaller than $\rho_0 $ in Theorem \ref{thm:locdefo}
and appyling the globalization Lemma \ref{lem:shorthom}, we deduce that $Y$ is CAT($\bar\kappa$).

It remains to deal with the case $\lambda \leq 0$ and 	$\kappa -4 \lambda \leq 0$.   But then $\kappa \leq 0$, hence $X$ is simply connected. Since $X$ is homeomorphic to $Y$, we deduce
from the theorem of Cartan--Hadamard that $Y$ is CAT($\bar\kappa$).
\end{proof}

\subsection{Conclusions}
We can now easily prove:

\begin{thm}  \label{thm: nonpos}
Let $x$ be a point in  a CAT(0) space $X$. Define the function
$f:X\to \R$ by  $f(y):=\frac 1  2 d^2(x,y)$. Then  the space $Y= e^f \cdot X$
is  CAT(0). Moreover, for any $R>0$, the closed ball $\bar B_R(x)$ around  $x$ in $Y$ is CAT($\kappa$) for some
$\kappa =\kappa (R) <0$.
\end{thm}

\begin{proof}
The function $f$ is  $1$-convex and Lipschitz continuous on bounded balls.
Thus,  $Y$  is CAT(0) by \cite{LS}.

By Lemma \ref{lem: distfunc}, the closed ball $\bar B_R (x)$ in $Y$ has the form $e^f\sd B$, where $B$ is the closed ball $\bar B_r (x)$  in $X$ and $r(R)$ is  such that
$$\int _0  ^r e^{\frac 1 2 t^2} \, dt =R \;.$$
From Theorem \ref{thm: confchange} we deduce that  $e^f\sd B$ is
CAT($\ka$) with
$$\kappa = - 4\cdot e^{-r^2} \;.$$
\end{proof}

Finally we can provide
\begin{proof}[Proof of Theorem \ref{thm: main}]

Clearly, we may assume $\kappa >0$ and, by rescaling, even $\kappa =1$.

Thus let $X$ be a CAT(1) space and let $O=B_r(x)$ be an open ball in $X$ with $r< \frac {\pi} {2}$.

We can replace $X$ by the closed ball $\bar B_r (x)$.
In order to simplify the calculation we proceed as follows. First we improve the curvature bound on the closed ball to $0$. In a second step, we change the metric on the open ball, to make it complete  and simultaneously
decrease the upper curvature bound to $-1$.

There exists $A >0$ depending only on  $r$, such that the function $g(y)=A\cdot  d^2(x,y)$ is  $1$-convex on $X$.
Due to Theorem \ref{thm: confchange}, the space
$Z= e^g \sd X$ is a CAT(0) space.  Moreover, by Lemma \ref{lem: distfunc}, the subset $e^g\sd O \subset  Z$ is an open ball in $Z$ around the point $x$.
Replacing the space $X$ by $Z$ we have reduced our task to the case $\kappa =0$. In this case the function $g(y)=\frac 1 2 d^2(x,y)$ is $1$-convex on $X$.

Now consider  the function $h:[0,\frac{r^2}{2}) \to \R$ given by
$$h(t)=-\log ( \frac{r^2}{2}-t)  \;.$$
Then the function $h$ is convex and $\lim\limits_{t\to \frac 1 2 r^2} h(t) =\infty$.  Moreover,
$$h'(t)=e^{h(t)} \;.$$
Consider the locally Lipschitz continuous function $f(y):= h(g(y))$ on $O=B_{r}(x)$ and the space $Y= e^f\sd O$.

For an arbitrary point $y\in O$,
we choose a small closed  ball $U$ around $y$, such that  for all $z\in U$ holds
\[h'(g(z))\geq\frac 1 2 h'(g(y))=\frac{1}{2}e^{h(g(y))}\ \ \text{ and }\ \ h(g(z))\leq 2h(g(y)).\]

 Due to the convexity of $h$ and the $1$-convexity of $g$, the
restriction of $f$ to any geodesic $\gamma$ in $O$ is at least $\lambda$-convex, where $\lambda$ denotes the minimum of $h'$ on the image
$g(\gamma)\subset [0, \frac{r^2}{2})$.

Hence for any such ball $U$, the space $e^f\sd U$ is CAT(-1), by Theorem \ref{thm: confchange}.  This shows that the space $Y$ is locally
CAT(-1).

  For any $s<r$ we deduce
from Lemma \ref{lem: distfunc}, that the subset $e^f\sd \bar B_s(x) \subset e^f \sd O$ coincides with the closed ball in $Y$ around $x$ of radius
\[R(s)=\int _0 ^s h(\frac{1}{2}t^2) \, dt=-\int _0 ^s \log(\frac{r^2-t^2}{2}) \, dt \]
Moreover, this ball is CAT(0) by Theorem \ref{thm: confchange}.  Since
$R(s)$ converges to infinity as $s$  converges to $r$, we deduce that
$Y$ is CAT(0). In particular, it is complete, simply connected and geodesic. Since we have already seen that $Y$ is locally CAT(-1), this finishes the proof.
\end{proof}

\bibliographystyle{alpha}
\bibliography{Change}

\newcommand{\etalchar}[1]{$^{#1}$}
\begin{thebibliography}{BFH{\etalchar{+}}16}

\bibitem[AB04]{AB}
S.~Alexander and R.~Bishop.
\newblock Curvature bounds for wrapped products in metrics spaces.
\newblock {\em Geom. Funct. Anal.}, 14:1143--1181, 2004.

\bibitem[AGS05]{AGS}
Luigi Ambrosio, Nicola Gigli, and Giuseppe Savar\'{e}.
\newblock {\em Gradient flows in metric spaces and in the space of probability
  measures}.
\newblock Lectures in Mathematics ETH Z\"{u}rich. Birkh\"{a}user Verlag, Basel,
  2005.

\bibitem[AKP16]{AKP}
S.~Alexander, V.~Kapovitch, and A.~Petrunin.
\newblock Alexandrov geometry.
\newblock {\em Preprint, http://anton-petrunin.github.io/book/all.pdf}, 2016.

\bibitem[Bal04]{Ballmann}
W.~Ballmann.
\newblock On the geometry of metric spaces.
\newblock {\em Preprint, lecture notes,
  http://people.mpim-bonn.mpg.de/hwbllmnn/archiv/sin40827.pdf}, 2004.

\bibitem[BBI01]{BBI01}
D.~Burago, Y.~Burago, and S.~Ivanov.
\newblock {\em A course in metric geometry}, volume~33 of {\em Graduate Studies
  in Mathematics}.
\newblock American Mathematical Society, Providence, RI, 2001.

\bibitem[Ber83]{Berest}
V.~Berestovskii.
\newblock Borsuk's problem on metrization of a polyhedron.
\newblock {\em Dokl. Akad. Nauk SSSR}, 268(2):273--277, 1983.

\bibitem[BFH{\etalchar{+}}16]{Mese-Breiner-regular}
C.~Breiner, A.~Fraser, L.~Huang, C.~Mese, P.~Sargent, and Y.~Zhang.
\newblock Regularity of harmonic maps from polyhedra to $cat(1)$ spaces.
\newblock {\em arXiv: 1610.07829}, 2016.

\bibitem[BH99]{BH}
M.~Bridson and A.~Haefliger.
\newblock {\em Metric spaces of non-positive curvature}, volume 319 of {\em
  Grundlehren der Mathematischen Wissenschaften}.
\newblock Springer-Verlag, Berlin, 1999.

\bibitem[Che95]{Chen}
J.~Chen.
\newblock On energy minimizing mappings between and into singular spaces.
\newblock {\em Duke Math. J.}, 79(1):77--99, 1995.

\bibitem[DM10]{Mese-general}
G.~Daskalopoulos and C.~Mese.
\newblock Harmonic maps between singular spaces {I}.
\newblock {\em Comm. Anal. Geom.}, 18(2):257--337, 2010.

\bibitem[Fug05]{Fuglede}
B.~Fuglede.
\newblock The {D}irichlet problem for harmonic maps from {R}iemannian polyhedra
  to spaces of upper bounded curvature.
\newblock {\em Trans. Amer. Math. Soc.}, 357(2):757--792, 2005.

\bibitem[Fug08]{Fuglede-reg}
B.~Fuglede.
\newblock Harmonic maps from {R}iemannian polyhedra to geodesic spaces with
  curvature bounded from above.
\newblock {\em Calc. Var. Partial Differential Equations}, 31(1):99--136, 2008.

\bibitem[GW17]{Guo}
C.~Guo and S.~Wenger.
\newblock Area minimizing discs in locally non-compact metric spaces.
\newblock {\em Comm. Anal. Geom., to appear, preprint}, 2017.

\bibitem[HKST15]{HKST15}
J.~Heinonen, P.~Koskela, N.~Shanmugalingam, and J.~Tyson.
\newblock {\em Sobolev spaces on metric measure spaces}, volume~27 of {\em New
  Mathematical Monographs}.
\newblock Cambridge University Press, Cambridge, 2015.

\bibitem[Ish79]{Ishihara}
T\^{o}ru Ishihara.
\newblock A mapping of {R}iemannian manifolds which preserves harmonic
  functions.
\newblock {\em J. Math. Kyoto Univ.}, 19(2):215--229, 1979.

\bibitem[Ken91]{Kendall}
W.~Kendall.
\newblock Convexity and the hemisphere.
\newblock {\em J. London Math. Soc. (2)}, 43(3):567--576, 1991.

\bibitem[KL97]{KleinerLeeb}
B.~Kleiner and B.~Leeb.
\newblock Rigidity of quasi-isometries for symmetric spaces and {E}uclidean
  buildings.
\newblock {\em Inst. Hautes \'{E}tudes Sci. Publ. Math.}, 86:115--197, 1997.

\bibitem[Kle99]{Kleiner}
B.~Kleiner.
\newblock The local structure of spaces with curvature bounded above.
\newblock {\em Math. Z.}, 231:409--456, 1999.

\bibitem[Kra11]{Linus}
L.~Kramer.
\newblock On the local structure and homology of {$CAT(\kappa )$} spaces and
  buidlings.
\newblock {\em Advances in Geometry}, 11:347--369, 2011.

\bibitem[KS93]{KS}
N.~Korevaar and R.~Schoen.
\newblock Sobolev spaces and harmonic maps for metric space targets.
\newblock {\em Comm. Anal. Geom.}, 1(3-4):561--659, 1993.

\bibitem[LN18]{LN2}
A.~Lytchak and K.~Nagano.
\newblock Topological regularity of spaces with an upper curvature bound.
\newblock {\em arXiv:1809.06183}, 2018.

\bibitem[LN19]{LN1}
A.~Lytchak and K.~Nagano.
\newblock Geodesically complete spaces with an upper curvature bound.
\newblock {\em Geom. Funct. Anal.}, 29:295--342, 2019.

\bibitem[LS17]{LS}
A.~Lytchak and S.~Stadler.
\newblock Conformal deformation of {$CAT(0)$} spaces.
\newblock {\em Math. Ann.}, Online first, 2017.

\bibitem[LW17]{LWplateau}
A.~Lytchak and S.~Wenger.
\newblock Area minimizing discs in metric spaces.
\newblock {\em Arch. Ration. Mech. Anal.}, 223(3):1123--1182, 2017.

\bibitem[LW18]{LWcurv}
Alexander Lytchak and Stefan Wenger.
\newblock Isoperimetric characterization of upper curvature bounds.
\newblock {\em Acta Math.}, 221(1):159--202, 2018.

\bibitem[Lyt05]{Ly-open}
A.~Lytchak.
\newblock Open map theorem for metric spaces.
\newblock {\em Algebra i Analiz}, 17(3):139--159, 2005.

\bibitem[May98]{Mayer}
U.~Mayer.
\newblock Gradient flows on nonpositively curved metric spaces and harmonic
  maps.
\newblock {\em Comm. Anal. Geom.}, 6(2):199--253, 1998.

\bibitem[Mes01]{Mese-curvature}
C.~Mese.
\newblock The curvature of minimal surfaces in singular spaces.
\newblock {\em Comm. Anal. Geom.}, 9(1):3--34, 2001.

\bibitem[MT02]{uniform}
R.~Mazzeo and M.~Taylor.
\newblock Curvature and uniformization.
\newblock {\em Israel J. Math.}, 130:323--346, 2002.

\bibitem[Oht09]{Ohta-2}
S.~Ohta.
\newblock Gradient flows on {W}asserstein spaces over compact {A}lexandrov
  spaces.
\newblock {\em Amer. J. Math.}, 131(2):475--516, 2009.

\bibitem[OP17]{Ohta}
S.~Ohta and M.~P\'{a}lfia.
\newblock Gradient flows and a {T}rotter-{K}ato formula of semi-convex
  functions on {${\rm CAT}(1)$}-spaces.
\newblock {\em Amer. J. Math.}, 139(4):937--965, 2017.

\bibitem[Pet07]{Petruninsemi}
A.~Petrunin.
\newblock Semiconcave functions in {A}lexandrov's geometry.
\newblock In {\em Surveys in differential geometry. {V}ol. {XI}}, pages
  137--201. Int. Press, Somerville, MA, 2007.

\bibitem[Res93]{Reshetnyak-GeomIV}
Yu.~G. Reshetnyak.
\newblock Two-dimensional manifolds of bounded curvature.
\newblock In {\em Geometry, {IV}}, volume~70 of {\em Encyclopaedia Math. Sci.},
  pages 3--163. Springer, Berlin, 1993.

\bibitem[Ser95]{Serbinowski}
T.~Serbinowski.
\newblock Harmonic maps into metric spaces of curvature bounded from above.
\newblock {\em Thesis, University of Utah}, 1995.

\bibitem[Yok16]{Yokota}
T.~Yokota.
\newblock Convex functions and barycenter on {CAT}(1)-spaces of small radii.
\newblock {\em J. Math. Soc. Japan}, 68(3):1297--1323, 2016.

\end{thebibliography}

\end{document}